\documentclass[12pt,reqno]{amsart}
\usepackage{amsmath}
\usepackage[dvips]{graphicx}
\usepackage{amsfonts}
\usepackage{amssymb}
\usepackage{latexsym}
\graphicspath{{Img/}}
\newtheorem{theorem}{Theorem}
\theoremstyle{plain}

\newtheorem{corollary}{Corollary}

\newtheorem{lemma}{Lemma}

\newtheorem{problem}{Problem}

\newtheorem{remark}{Remark}

\numberwithin{equation}{section}

\begin{document}

\title[New Solutions for The weighted Fermat-Torricelli problem]{Two new analytical solutions and two new geometrical solutions for the weighted Fermat-Torricelli problem in the Euclidean Plane}
\author{Anastasios N. Zachos}
\address{University of Patras, Department of Mathematics, GR-26500 Rion, Greece}
\email{azachos@gmail.com} 

\begin{abstract}
We obtain two analytic solutions for the weighted
Fermat-Torricelli problem in the Euclidean Plane which states
that: Given three points in the Euclidean plane and a positive
real number (weight) which correspond to each point, find the
point such that the sum of the weighted distances to these three
points is minimized.
 Furthermore, we give two new geometrical solutions  for the
the weighted Fermat-Torricelli problem (weighted Fermat-Torricelli
point), by using the floating equilibrium condition of the
weighted Fermat-Torricelli problem (first geometric solution) and
a generalization of Hofmann's rotation proof under the condition
of equality of two given weights (second geometric solution).
\end{abstract}\maketitle

\section{Introduction}

We state the weighted Fermat-Torricelli problem in
$\mathbb{R}^{2}:$

\begin{problem}
Given a triangle $\triangle A_{1}A_{2}A_{3}$ with vertices
$A_{1}=(x_{1},y_{1}),$ $A_{2}=(x_{2},y_{2}),$
$A_{3}=(x_{3},y_{3}),$ find a fourth point $A_{F}=(x_{F},y_{F})$
which minimizes the objective function

\begin{equation}\label{obj1}
f(x,y)=\sum_{i=1}^{3}B_{i}\sqrt{(x-x_{j})^{2}+(y-y_{j})^{2}}
\end{equation}

where $B_{i}$ is a positive real number (weight) which corresponds
to $A_{i}.$

\end{problem}

By replacing $B_{1}=B_{2}=B_{3}$ in (\ref{obj1}), we obtain the
(unweighted) Fermat-Torricelli problem which was first stated by
Pierre de Fermat (1643).

The solution of the weighted Fermat-Torricelli problem (Problem~1)
is called the weighted Fermat-Torricelli point $A_{F}.$

The existence and uniqueness of the weighted Fermat-Torricelli
point and a complete characterization of the "floating case" and
"absorbed case" has been established by Y. S Kupitz and H. Martini
(see \cite{Kup/Mar:97}, theorem 1.1, reformulation 1.2 page 58,
theorem 8.5 page 76, 77). A particular case of this result for
three non-collinear points in $\mathbb{R}^{2},$ is given by the
following theorem:

\begin{theorem}{\cite{BolMa/So:99},\cite{Kup/Mar:97}}\label{theor1}
Let there be given a triangle $\triangle A_{1}A_{2}A_{3},$ $A_{1},
A_{2}, A_{3}\in\mathbb{R}^{2}$ with corresponding positive
weights $B_{1}, B_{2}, B_{3}.$ \\
(a) The weighted Fermat-Torricelli point $A_{F}$ exists and is
unique. \\
(b) If for each point $A_{i}\in\{A_{1},A_{2},A_{3}\}$

\begin{equation}\label{floatingcase}
\|{\sum_{j=1, i\ne j}^{3}B_{j}\vec u(A_i,A_j)}\|>B_i,
\end{equation}

 for $i,j=1,2,3$  holds,
 then \\
 ($b_{1}$) the weighted Fermat-Torricelli point $A_{F}$ (weighted floating equilibrium point) does not belong to $\{A_{1},A_{2},A_{3}\}$
 and \\
 ($b_{2}$)

\begin{equation}\label{floatingequlcond}
 \sum_{i=1}^{3}B_{i}\vec u(A_F,A_i)=\vec 0,
\end{equation}
where $\vec u(A_{k} ,A_{l})$ is the unit vector from $A_{k}$ to
$A_{l},$ for $k,l\in\{0,1,2,3\}$
 (Weighted Floating Case).\\
 (c) If there is a point $A_{i}\in\{A_{1},A_{2},A_{3}\}$
 satisfying
 \begin{equation}
 \|{\sum_{j=1,i\ne j}^{3}B_{j}\vec u(A_i,A_j)}\|\le B_i,
\end{equation}
then the weighted Fermat-Torricelli point $A_{F}$ (weighted
absorbed point) coincides with the point $A_{i}$ (Weighted
Absorbed Case).
\end{theorem}
%-----------------------------------------------------------------------------------------------------------------------
A direct consequence of the weighted floating case and the
weighted absorbed case of theorem~\ref{theor1} gives
Corollary~\ref{cor1} (Torricelli's theorem) and
Corollary~\ref{cor2} (Cavalieri's alternative) (see in
\cite{BolMa/So:99}, p~236).
\begin{corollary}\label{cor1}
If $B_{1}=B_{2}=B_{3}$ and all three angles of the triangle
$\triangle A_{1}A_{2}A_{3}$  are less than $120^{\circ},$ then
$A_{F}$ is the isogonal point (interior point) of $\triangle
A_{1}A_{2}A_{3}$ whose sight angle to every side of
$A_{1}A_{2}A_{3}$ is $120^{\circ}$ (Torricelli's theorem).
\end{corollary}\label{cor2}
\begin{corollary}If $B_{1}=B_{2}=B_{3}$ and one of the angles of
the triangle $\triangle A_{1}A_{2}A_{3}$ is equal or greater than
$120^{\circ},$ then $A_{F}$ is the vertex of the obtuse angle of
$\triangle A_{1}A_{2}A_{3}$ (Cavalieri's alternative).
\end{corollary}

%--------------------------------------------------------------------

%We state the  classical weighted Fermat-Torricelli problem:
%\begin{problem}
%Consider a triangle $\triangle A_1A_2A_3$ with positive constant
%weights $w_i$ that correspond to the vertex $A_{i},$ for
%$i=1,2,3.$ Find a point $A_{0}$ for which the sum
%\begin{equation} \label{eq:BBB}
%w_1a_1+w_2a_2+w_3a_3
%\end{equation}
%is minimized.
%\end{problem}

Concerning the solution of the weighted Fermat-Torricelli problem
with the use of analytic geometry and trigonometry, we mention the
works of \cite{GreenbergRobertello:65}, \cite{Jones:88},
\cite{Pesamosca:91} \cite{MHAjja:94}, \cite{Gue/Tes:02},
\cite{Zachos/Zou:08} and \cite{Zachos:13}.

Recently, an analytic solution, which express explicitly the
coordinates of the weighted Fermat-Torricelli point with respect
to the coordinates of the three points $A_{i}$ and the three
weights $B_{1},$ $B_{2}$ $B_{3}$ for the weighted
Fermat-Torricelli problem with respect to the weighted floating
case of Theorem~1 has been derived in \cite{Uteshev:12} and for
the case $B_{1}=B_{2}=B_{3}=1$ has been derived in
\cite{Roussos:12}.

In this paper, we present two new analytical solutions for the
weighted Fermat-Torricelli problem in $\mathbb{R}^{2}$ in the
weighted floating case of Theorem~\ref{theor1}. The first
analytical solution gives the coordinates of the weighted
Fermat-Torricelli point as a function of the coordinates of the
three non collinear points and the three given weights (real
positive numbers) in a different way from \cite{Uteshev:12} and
\cite{Roussos:12} (Theorem~\ref{analytic1}. Section~2).

The second analytical solution gives the location of the weighted
Fermat-Torricelli point as a function of two inscribed angles of
the circumscribed circle which passes form the three non collinear
points and the three given weights by applying a coordinate
independent approach given in \cite{Zachos:13}
(Theorem~\ref{angular}, Section~3).

The first geometrical solution of the weighted Fermat-Torricelli
point with ruler and compass focuses on constructing the
intersection of two Simpson lines (weighted case) by applying the
duality of the weighted Fermat-Torricelli problem which was
introduced in \cite{Uteshev:12} (Problem~\ref{firstprob},
Section~4).

Finally, the second geometric solution of the weighted
Fermat-Torricelli point focuses on finding the angle of rotation
of the three non collinear points about one of them and
generalizes Hofmann's rotation proof (\cite{BolMa/So:99},
\cite{Hofmann:27}, \cite{Tikh:86}) regarding the equality of the
given weights (Problem~\ref{secgeomprob}, Corollary~\ref{corhof},
Section~4).

%--------------------------------------------------------------------------------------------------------

\section{Analytical solution of the weighted Fermat-Torricelli problem}

We obtain an analytic solution for the floating case of Theorem~1,
i.e the weighted Fermat-Torricelli point $A_{F}$ is an interior
point of $\triangle A_{1}A_{2}A_{3},$ such that the coordinates
$x_{F}$ and $y_{F}$ of $A_{F}$ are expressed explicitly as a
function of $x_{i},y_{i}$ and $B_{i},$ for $i=1,2,3,$ by using
analytic geometry in $\mathbb{R}^{2}.$

We denote by $a_{ij}$ the length of the linear segment $A_iA_j$
and $\alpha_{ikj}$ the angle $\angle A_{i}A_{k}A_{j}$ for
$i,j,k=0,1,2,2',3,3', i\neq j\neq k$ (See fig.~\ref{fig1}).

Without loss of generality, we set $A_{1}=(0,0),$
$A_{2}=(a_{12},0),$ $A_{3}=(x_{3},y_{3}).$

We need the following two lemmata:

\begin{lemma}{\cite{Gue/Tes:02},\cite{Zachos/Zou:08}}\label{lem1}
Under the condition (\ref{floatingcase}) and the weighted floating
equilibrium condition (\ref{floatingequlcond}) the following
equation is satisfied:

\begin{equation}\label{sinlawweights}
\frac{B_{i}}{\sin\alpha_{203}}=\frac{B_{2}}{\sin\alpha_{103}}=\frac{B_{3}}{\sin\alpha_{102}}=C,
\end{equation}

where
$C=\frac{2B_{1}B_{2}B_{3}}{\sqrt{(B_{1}+B_{2}+B_{3})(B_{2}+B_{3}-B_{1})(B_{1}+B_{3}-B_{2})(B_{1}+B_{2}-B_{3})}}$

\end{lemma}

\begin{lemma}{\cite{BolMa/So:99},\cite{MHAjja:94},\cite{Gue/Tes:02},\cite{Zachos/Zou:08}}\label{lem2}
Under the condition (\ref{floatingcase}) and the weighted floating
equilibrium condition (\ref{floatingequlcond}) the angle
$\alpha_{i0j}$ is expressed as a function of $B_{1}, B_{2}$ and
$B_{3}:$

\begin{equation}\label{alphai0j}
\alpha_{i0j}=\arccos\left(\frac{B_{k}^2-B_{i}^2-B_{j}^2}{2B_{i}B_{j}}\right)
\end{equation}
for $i,j,k=1,2,3,$ and $i\ne j\ne k.$

\end{lemma}

\begin{theorem}\label{analytic1}
Under the condition (\ref{floatingcase}) of the weighted floating
case, the coordinates of the weighted Fermat-Torricelli point
$A_{F}$ $(x_{F},y_{F})$ are given by the following relations:

\begin{equation}\label{xanalytic}
x_{F}=-\frac{(a_{12}-x_{2}')(x_{3}'y_{3}-x_{3}y_{3}')+d_{3}(x_{3}-x_{3}')}{(a_{12}-x_{2}')(y_{3}'-y_{3})-(y_{3}-x_{3}')y_{2}'},
\end{equation}

\begin{equation}\label{yanalytic}
y_{F}=\frac{y_{2}'(a_{12}y_{3}-x_{3}'y_{3}-a_{12}y_{3}'+x_{3}y_{3}')}{x_{3}y_{2}'-x_{3}y_{2}'+a_{12}y_{3}-x_{2}'y_{3}-a_{12}y_{3}'+x_{2}'y_{3}'}
\end{equation}

where

\begin{equation}\label{x2prime}
x_{2}'=-\frac{B_{3}}{B_{2}}\left(x_{3}\frac{B_{1}^2-B_{2}^2-B_{3}^2}{2B_{2}B_{3}}+y_{3}\sqrt{1-\left(\frac{B_{1}^2-B_{2}^2-B_{3}^2}{2B_{2}B_{3}}\right)^{2}},
\right)
\end{equation}

\begin{equation}\label{y2prime}
y_{2}'=\frac{B_{3}}{B_{2}}\left(x_{3}\sqrt{1-\left(\frac{B_{1}^2-B_{2}^2-B_{3}^2}{2B_{2}B_{3}}\right)^{2}}-y_{3}\frac{B_{1}^2-B_{2}^2-B_{3}^2}{2B_{2}B_{3}}\right),
\end{equation}

\begin{equation}\label{x3prime}
x_{3}'=-a_{12}\frac{B_{2}}{B_{3}}\left(\frac{B_{1}^2-B_{2}^2-B_{3}^2}{2B_{2}B_{3}}\right)
\end{equation}
and
\begin{equation}\label{y3prime}
y_{3}'=-a_{12}\frac{B_{2}}{B_{3}}\left(\sqrt{1-\left(\frac{B_{1}^2-B_{2}^2-B_{3}^2}{2B_{2}B_{3}}\right)^{2}}\right).
\end{equation}

\end{theorem}

\begin{proof}[Proof of Theorem~\ref{analytic1}:]

\begin{figure}
\centering
\includegraphics[scale=0.8]{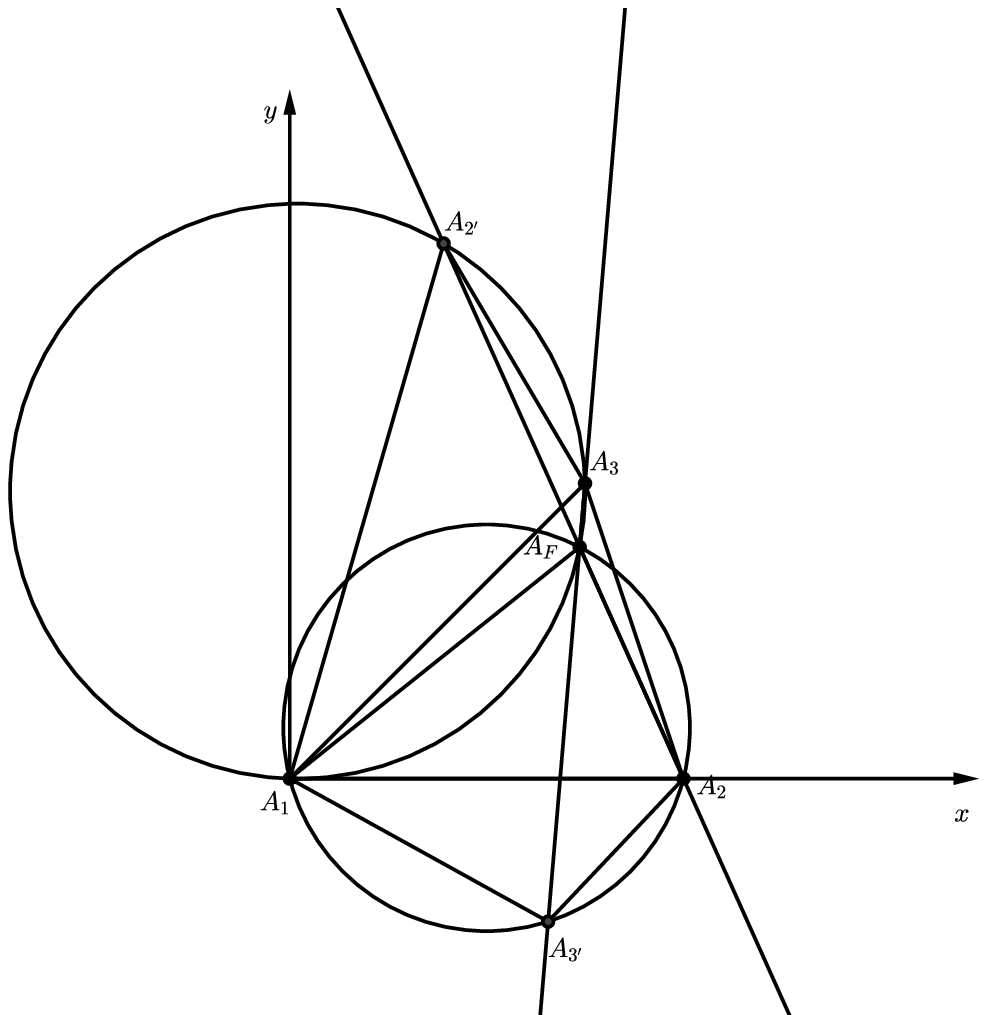}
\caption{}\label{fig1}
\end{figure}

We apply the weighted Torricelli configuration which is similar to
the configuration used in \cite{ENGELBRECHt:1877} and we construct
two similar triangles $\triangle A_{1}A_{2}A_{3'}$ and $\triangle
A_{1}A_{3}A_{2'},$ such that:
\begin{equation}\label{alpha13p2132p}
\alpha_{13^{\prime}2}=\alpha_{132^{\prime}}=\pi-\alpha_{102},
\end{equation}
\begin{equation}\label{alpha13p2132p}
\alpha_{213^{\prime}}=\alpha_{312^{\prime}}=\pi-\alpha_{203},
\end{equation}
and
\begin{equation}\label{alpha13p2132p}
\alpha_{123^{\prime}}=\alpha_{12^{\prime}3}=\pi-\alpha_{103}.
\end{equation}

From (\ref{alpha13p2132p}), (\ref{alpha13p2132p}) and
(\ref{alpha13p2132p}), we derive that the point of intersection of
the two circles which pass from $A_{1}, A_{3^{\prime}}, A_{2}$ and
$A_{1}, A_{2^{\prime}}, A_{3},$ respectively, is the weighted
Fermat-Torricelli point $A_{F}$ (fig.~\ref{fig1}). Therefore, we
obtain that $A_{F}$ is the intersection point of the lines
(weighted Simpson lines) defined by $A_{2}A_{2^{\prime}}$ and
$A_{3}A_{3^{\prime}}.$

Thus, we have:

\begin{equation}\label{eqanal1}
x_{2}'=a_{12'}\cos(\alpha_{213}+\pi-\alpha_{203}),
\end{equation}

\begin{equation}\label{eqanal2}
y_{2}'=a_{12'}\sin(\alpha_{213}+\pi-\alpha_{203}),
\end{equation}

\begin{equation}\label{eqanal3}
x_{3}'=a_{13'}\cos(\pi-\alpha_{203}),
\end{equation}

and

\begin{equation}\label{eqanal4}
y_{3}'=-a_{13'}\sin(\pi-\alpha_{203}).
\end{equation}

%---------------------------------------------------------------------------------

By applying the sine law in $\triangle A_{1}A_{2}A_{3'}$ and
$\triangle A_{1}A_{3}A_{2'},$ we get, respectively,:

\begin{equation}\label{eqanal5}
a_{13'}=a_{12}\frac{B_{2}}{B_{3}}.
\end{equation}

and

\begin{equation}\label{eqanal6}
a_{12'}=a_{13}\frac{B_{3}}{B_{2}}.
\end{equation}

By replacing (\ref{eqanal5}),(\ref{eqanal6}) and (\ref{alphai0j})
from lemma~\ref{lem2} in (\ref{eqanal1}), (\ref{eqanal2}),
(\ref{eqanal3}) and (\ref{eqanal4}), we obtain (\ref{x2prime}),
(\ref{y2prime}), (\ref{x3prime}) and (\ref{y3prime}).

The equations of the lines defined by $A_{2}A_{2'}$ and
$A_{3}A_{3'},$ respectively, are as follows:

\begin{equation}\label{eqanal7}
\frac{y_{2}'}{x_{2}'-a_{12}}=\frac{y}{x-a_{12}}
\end{equation}
and
\begin{equation}\label{eqanal8}
\frac{y_{3}-y_{3}'}{x_{3}-x_{3}'}=\frac{y_{3}-y}{x_{3}-x}
\end{equation}

Solving (\ref{eqanal7}) and (\ref{eqanal8}) with respect to
$(x,y)$ we derive the point of intersection $A_{F}=(x_{F},y_{F}),$
and the coordinates $x_{F}$ and $y_{F}$ are given by
(\ref{xanalytic}) and (\ref{yanalytic}), respectively.

\end{proof}

\section{An explicit angular solution of the weighted Fermat-Torricelli problem}

It is well known that the barycenter $A_{m}$ of $\triangle
A_{1}A_{2}A_{3}$ is constructed by the relation
$a_{im}=\frac{2}{3}a_{i,jk}=\frac{1}{3}\sqrt{2a_{ij}^2+2a_{ik}^2-a_{jk}^2},$
where $a_{i,jk}$ is the length of the midline that connects the
vertex $A_{i}$ with the midpoint of the line segment $A_{j}A_{k}$
for $i,j,k=1,2,3,$ $i\ne j\ne k$ and the median minimizes the
objective function $a_{m1}^{2}+a_{m2}^{2}+a_{m3}^{2}.$

A natural question to ask is if the location of the weighted
Fermat-Torricelli problem could be given with respect to the
lengths of the sides of $\triangle A_{1}A_{2}A_{3}$ and the
constant positive weights $B_{1}, B_{2},B_{3}.$

A positive answer to this question is given by the following lemma
(\cite[Corollary~2]{Zachos:13}):

\begin{lemma}{(\cite[Corollary~2]{Zachos:13})}\label{lem3}The explicit solution of the weighted Fermat-Torricelli problem in $\mathbb{R}^{2}$, under the condition
(\ref{floatingcase}) (weighted floating case)  is given by:
\begin{equation}\label{mainres}
\alpha_{013}=\operatorname{arccot}\left(\frac{\sin(\alpha_{213})-\cos(\alpha_{213})
\cot(\arccos{\frac{B_{3}^{2}-B_{1}^2-B_{2}^2}{2B_{1}B_{2}}})-
\frac{a_{13}}{a_{12}
}\cot(\arccos{\frac{B_{2}^{2}-B_{1}^2-B_{3}^2}{2B_{1}B_{3}}})}
{-\cos(\alpha_{213})-\sin(\alpha_{213})
\cot(\arccos{\frac{B_{3}^{2}-B_{1}^2-B_{2}^2}{2B_{1}B_{2}}})+
\frac{a_{13}}{a_{12} }}\right)
\end{equation}
and
\begin{equation}\label{mainresb}
a_{10}=\frac{\sin\left(\alpha_{013}+\arccos{\frac{B_{2}^{2}-B_{1}^2-B_{3}^2}{2B_{1}B_{3}}}\right)a_{13}}{\sin\left(\arccos{\frac{B_{2}^{2}-B_{1}^2-B_{3}^2}{2B_{1}B_{3}}}\right)}.
\end{equation}
where

\begin{equation}\label{mainresbr2bb}
\alpha_{213}=\arccos\left(\frac{a_{12}^2+a_{13}^2-a_{23}^2}{2a_{12}a_{13}}\right)
\end{equation}

and $\alpha_{013}$ and $a_{10}$ depend on $B_{1}, B_{2}, B_{3},$
$a_{13},$ $a_{12}$ and $a_{23}.$

\end{lemma}

\begin{proof}[Proof of Lemma~\ref{lem3}:]
We assume that the weighted floating case occurs (see
theorem~\ref{theor1}, Case~b), in order to locate it in the
interior of the $\triangle A_1A_2A_3.$

From the cosine law in $\triangle A_{1}A_{0}A_{2},$ and $\triangle
A_{1}A_{0}A_{3}$ we get, respectively:
\begin{equation}\label{i23cos}
a_{02}^{2}=a_{01}^{2}+a_{12}^{2}-2a_{01}a_{12}\cos(\alpha_{213}-\alpha_{013})
\end{equation}
and
\begin{equation}\label{i23cos3}
a_{03}^{2}=a_{01}^{2}+a_{13}^{2}-2a_{01}a_{13}\cos(\alpha_{013}).
\end{equation}
From (\ref{i23cos}) and (\ref{i23cos3}), $a_{02}$ and $a_{03}$ are
expressed with respect to the two variables $a_{01}$ and
$\alpha_{013}:$
\[a_{0i}=a_{0i}(a_1,\alpha_{013}),\]

for $i=2,3.$ By differentiating (\ref{obj1}) with respect to
$a_{01}$ and $\alpha_{013},$ respectively, we get:
\begin{equation}\label{der1}
B_1+B_2\frac{\partial a_{02}}{\partial a_{01}}+B_3\frac{\partial
a_{03}}{\partial a_{01}}=0,
\end{equation}
\begin{equation}\label{der2}
B_2\frac{\partial a_{02}}{\partial \alpha_{013}}+B_3\frac{\partial
a_{03}}{\partial \alpha_{013}}=0.
\end{equation}

From Appendix~A, by replacing (\ref{dera2a1}) and (\ref{dera3a1})
in (\ref{der1}), we obtain:

\begin{equation} \label{eqcos}
B_2\cos(\alpha_{102})+B_3\cos(\alpha_{103})=-B_{1}
\end{equation}

By replacing (\ref{dera2a013}) and (\ref{dera3a013}) in
(\ref{der2}), we obtain:
\begin{equation} \label{eqsin}
-B_2\sin(\alpha_{102})+B_3\sin(\alpha_{103})=0
\end{equation}

By squaring both parts of (\ref{eqcos}) and (\ref{eqsin}) and by
adding the two derived equations, we get:

\begin{equation}\label{cos203}
\cos(\alpha_{203})=\frac{B_{1}^2-B_{2}^2-B_{3}^2}{2B_{2}B_{3}}.
\end{equation}

Similarly by expressing the objective function with respect to the
two variables $a_{2},$ $\alpha_{023},$ and with respect to the two
variables $a_{3},$ $\alpha_{031},$ we derive, respectively:

\begin{equation}\label{cos103}
\cos(\alpha_{103})=\frac{B_{2}^2-B_{1}^2-B_{3}^2}{2B_{1}B_{3}}
\end{equation}
and
\begin{equation}\label{cos102}
\cos(\alpha_{102})=\frac{B_{3}^2-B_{1}^2-B_{2}^2}{2B_{1}B_{2}}.
\end{equation}

From the sine law in $\triangle A_1A_0A_2$, $\triangle A_1A_0A_3,$
we get, respectively:

\begin{equation}\label{sinl12}
\frac{a_{12}}{\sin(\alpha_{102})}=\frac{a_{01}}{\sin(\alpha_{213}-\alpha_{013}+\alpha_{102})}
\end{equation}
and
\begin{equation}\label{sinl13}
\frac{a_{13}}{\sin(\alpha_{103})}=\frac{a_{01}}{\sin(\alpha_{013}+\alpha_{103})}
\end{equation}

By eliminating $a_{01}$ from (\ref{sinl12}) and (\ref{sinl13}), we
obtain:
\begin{equation} \label{result}
\alpha_{013}=arccot(\frac{\sin(\alpha_{213})-\cos(\alpha_{213})
\cot(\alpha_{102})- \frac{a_{31}}{a_{12} }\cot(\alpha_{103})}
{-\cos(\alpha_{213})-\sin(\alpha_{213}) \cot(\alpha_{102})+
\frac{a_{31}}{a_{12} }})
\end{equation}
By replacing (\ref{cos102}) and (\ref{cos103}) in (\ref{result}),
we obtain (\ref{mainres}). From the sine law in $\triangle
A_{1}A_{0}A_{3},$ we derive (\ref{mainresb}).

The values of $a_{01}$ and $\alpha_{013}$ give the location of the
weighted Fermat-Torricelli point $A_{F}.$
\end{proof}

\begin{remark}
The explicit solution of the weighted Fermat-Torricelli problem is
similar with the definition of a complex number in a polar form:\\
$z=r\exp(i(\alpha_{213}-\alpha_{013})),$ where the absolute value
of z is $r=a_{1}$ and the argument of z is $\arg
z=\alpha_{213}-\alpha_{013}.$
\end{remark}

Let $C(Q,R)$ be the inscribed circle with center $Q$ and radius
$R$ which passes from the vertex $A_{i},$ for $i=1,2,3.$

Each of the three central angles is given by the relation:
\begin{equation}\label{centralij}
c_{iQj}=2\alpha_{imj},
\end{equation}
such that:
\[c_{1Q2}+c_{2Q3}+c_{1Q3}=2\pi\]
or
\begin{equation}\label{central1Q2}
c_{1Q2}=2\pi-c_{1Q3}-c_{2Q3},
\end{equation}
 for $i\ne m \ne j,$ $i,m,j=1,2,3.$ From the sine law in
$\triangle A_{1}A_{2}A_{3}$ and taking into account
(\ref{centralij}), we get:
\begin{equation}\label{centralijsine}
\frac{a_{13}}{\sin(\frac{c_{1Q3}}{2})}=\frac{a_{12}}{\sin(\frac{c_{1Q2}}{2})}=2R.
\end{equation}
By replacing (\ref{centralij}), (\ref{central1Q2})
(\ref{centralijsine}) in (\ref{mainres}) and (\ref{mainresb}) of
lemma~\ref{lem3}, we derive the following result:
\begin{theorem}\label{angular}
An explicit angular solution of the weighted Fermat-Torricelli
problem in $\mathbb{R}^{2},$ under the condition
(\ref{floatingcase}) is given by:
\begin{equation}\label{mainresang}
\cot\alpha_{013}=\frac{\sin(\frac{c_{2Q3}}{2})-\cos(\frac{c_{2Q3}}{2})
\cot(\arccos{\frac{B_{3}^{2}-B_{1}^2-B_{2}^2}{2B_{1}B_{2}}})-
\frac{\sin(\frac{c_{1Q3}}{2})}{\sin(\frac{c_{1Q3}+c_{2Q3}}{2})
}\cot(\arccos{\frac{B_{2}^{2}-B_{1}^2-B_{3}^2}{2B_{1}B_{3}}})}
{-\cos(\frac{c_{2Q3}}{2})-\sin(\frac{c_{2Q3}}{2})
\cot(\arccos{\frac{B_{3}^{2}-B_{1}^2-B_{2}^2}{2B_{1}B_{2}}})+
\frac{\sin(\frac{c_{1Q3}}{2})}{\sin(\frac{c_{1Q3}+c_{2Q3}}{2}) }}
\end{equation}
and
\begin{equation}\label{mainresbang}
a_{10}=2R\frac{\sin(\alpha_{013}+\arccos{\frac{B_{2}^{2}-B_{1}^2-B_{3}^2}{2B_{1}B_{3}}})
\sin(\frac{c_{1Q3}}{2})}{\sin(\arccos{\frac{B_{2}^{2}-B_{1}^2-B_{3}^2}{2B_{1}B_{3}}})},
\end{equation}
where  $\alpha_{013}$ and $a_{10}$ depend on $B_{1}, B_{2},
B_{3},$ $c_{1Q3},$ $c_{2Q3},$ and $R.$

\end{theorem}

\begin{remark} We conclude that by setting $R=1$ (unit radius circumscribed circle) in
(\ref{mainresbang}), the explicit solution depends only on five
given elements: $B_{1}, B_{2}, B_{3},$ $c_{1Q3}$ and $c_{2Q3}.$
This unique result holds only if the inequalities
(\ref{floatingcase}) of the weighted floating case of Theorem~1,
Case (b) are satisfied.
\end{remark}

\begin{remark}
We note that lemma~\ref{lem3} and theorem~\ref{angular} provide an
analytic solution for the weighted Fermat-Torricelli problem in
$\mathbb{R}^{2}$ without using the coordinates of the points
$A_{i},$ for $i=1,2,3$ (Coordinate independent approach) taking
into account only given Euclidean elements (lengths and angles).
\end{remark}

\section{Two new geometrical solutions of the weighted Fermat-Torricelli point in the Euclidean plane}

We present two new geometrical solutions to find the weighted
Fermat-Torricelli point in the weighted floating case.

The first solution deals with the position of $A_{1'}$ and
$A_{3'}$ which shall give the position of the two Simpson lines
defined by $A_{2}$ and $A_{2'}$ and $A_{3},$ $A_{3'}$
(fig.~\ref{fig2})

The second solution deals with the generalization of Hofmann's
rotation proof (\cite{Tikh:86},\cite{Hofmann:27}) for
$B_{1}=B_{2},$ such that $B_{1}$

\begin{problem}\label{firstprob}Construct the solution of Problem~1 (Weighted Fermat-Torricelli problem) using ruler and compass, under the condition (\ref{floatingcase}).
\end{problem}

\begin{figure}
\centering
\includegraphics[scale=0.8]{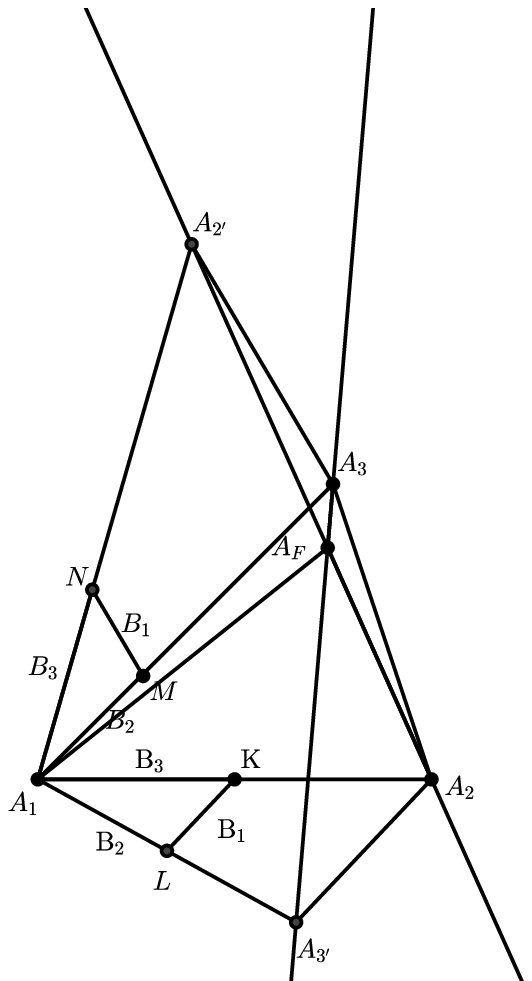}
\caption{}\label{fig2}
\end{figure}

\begin{proof}[Solution of Problem~\ref{firstprob}:]

We need to construct the vertices $A_{3'}$ and $A_{2'}$ (see
fig.~\ref{fig2}).

First, we construct the vertex $A_{3'}.$ We select a point $K$
which belongs to the linear segment $A_{1}A_{2},$ such that
$\|A_{1}K\|=B_{3}$ and we construct a triangle $\triangle A_{1}KL$
with the other two sides $\|A_{1}L\|=B_{2}$ and $\|KL\|=B_{3}.$

Thus, lemma~\ref{lem1} yields $\angle
A_{2}A_{1}A_{3'}=\alpha_{213'}=\pi-\alpha_{203}.$

We need to calculate $a_{13'},$ in order to find the location of
$A_{3}'.$

Take a point $G$ to the line defined by $A_{1}A_{2},$ such that
$\|A_{1}A_{2}\|=a_{12},$ $\|A_{2}G\|=B_{3}$ where
$\|A_{1}G\|=a_{12}+B_{3}$ and construct with a ruler and compass
the perpendicular linear segment at the point $G$ to the line
defined by  $A_{1}A_{2}$ and take a point $H$ such that
$\\|GH\|=B_{2}.$ We denote by $I$ the point of intersection of the
line defined by $A_{2}H$ and the perpendicular line at the point
$A_{1}$ with respect to the line defined by $A_{1}A_{2}$
(fig.~\ref{fig3}). Taking into account the similar triangles
$\triangle A_{1}A_{2}I$ and  $\triangle A_{2}GH,$ we get:

\begin{equation}\label{eqnalc1}
a_{13'}=a_{12}\frac{B_{2}}{B_{3}}.
\end{equation}

\begin{figure}
\centering
\includegraphics[scale=0.8]{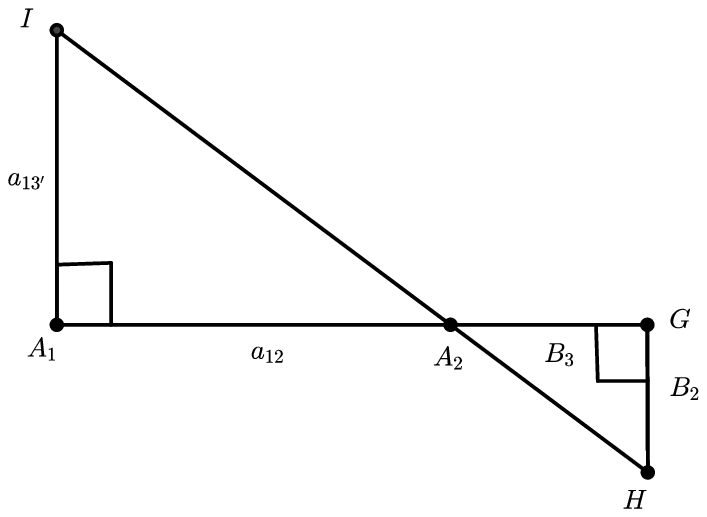}
\caption{}\label{fig3}
\end{figure}

Similarly, we construct the vertex $A_{2'}.$ We select a point $M$
which belongs to the linear segment $A_{1}A_{3},$ such that
$\|A_{1}M\|=B_{2}$ and we construct a triangle $\triangle A_{1}MN$
with the other two sides $\|A_{1}N\|=B_{3}$ and $\|MN\|=B_{1}.$

Thus, lemma~\ref{lem1} yields $\angle
A_{3}A_{1}A_{2'}=\alpha_{312'}=\pi-\alpha_{203}.$

We need to calculate $a_{12'},$ in order to find the location of
$A_{2}'.$

Take a point $Q$ to the line defined by $A_{1}A_{3},$ such that
$\|A_{1}A_{3}\|=a_{13},$ $\|A_{2}G\|=B_{3}$ where
$\|A_{1}Q\|=a_{13}+B_{3}$ and construct with a ruler and compass
the perpendicular linear segment at the point $Q$ to the line
defined by  $A_{1}A_{3}$ and take a point $R$ such that
$\\|RQ\|=B_{3}.$ We denote by $P$ the point of intersection of the
line defined by $A_{3}R$ and the perpendicular line at the point
$A_{1}$ with respect to the line defined by $A_{1}A_{3}$
(fig.~\ref{fig4}). Taking into account the similar triangle
$\triangle A_{1}A_{3}P$ and  $\triangle A_{3}QR,$ we get:

\begin{equation}\label{eqnalc1}
a_{12'}=a_{13}\frac{B_{3}}{B_{2}}.
\end{equation}

\begin{figure}
\centering
\includegraphics[scale=0.8]{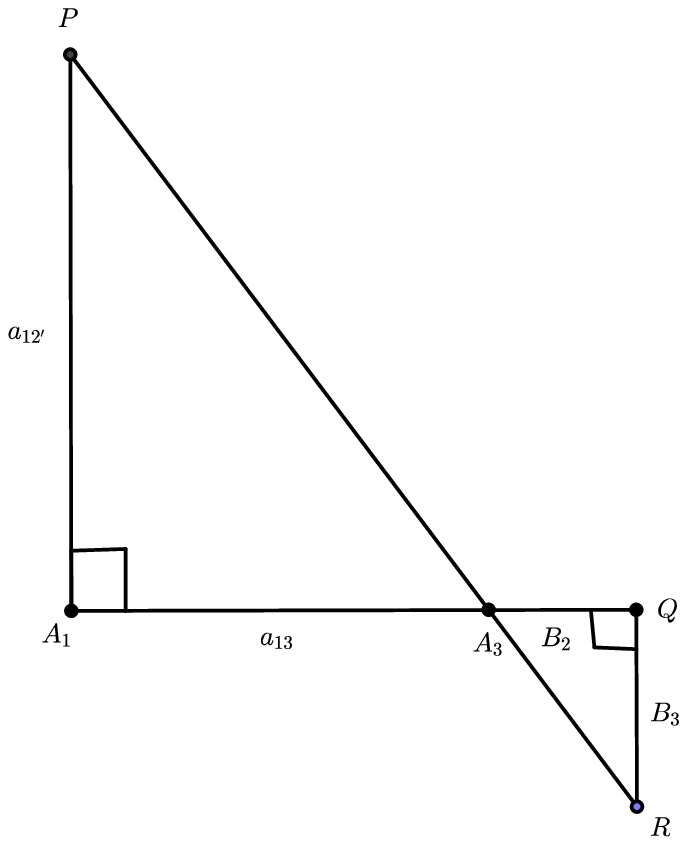}
\caption{}\label{fig4}
\end{figure}

\end{proof}

\begin{problem}\label{secgeomprob}Solve Problem~1 (Weighted Fermat-Torricelli problem)  by generalizing Hofmann's rotation, under the condition
(\ref{floatingcase}) and  $B_{1}=B_{2},$ such that
$B_{1}+B_{2}+B_{3}=1,$ $B_{1}>\frac{1}{4}$ and
$\alpha_{132}>\pi-\arccos\left(-1+\frac{(1-2B_{1})^{2}}{2B_{1}^{2}}\right).$
\end{problem}

\begin{proof}[Solution of Problem~\ref{secgeomprob}:]
\begin{figure}
\centering
\includegraphics[scale=0.8]{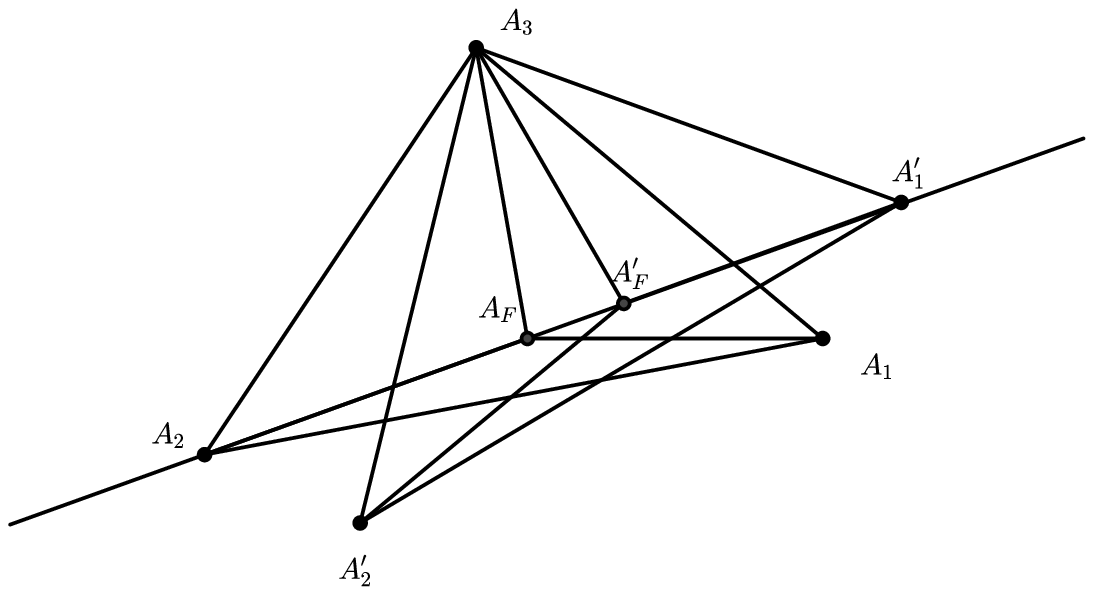}
\caption{}\label{fig5}
\end{figure}

We consider a weight $B_{i}$ which corresponds to the vertex
$A_{i}$ in $\mathbb{R}^{2},$ for $i=1,2,3.$

By replacing $B_{1}=B_{2}$ in (\ref{alphai0j}) of
lemma~\ref{lem2}, we derive that:

\begin{equation}\label{alpha203103}
\alpha_{203}=\alpha_{103}=\arccos\left(1-\frac{1}{2B_{1}}\right)
\end{equation}

and

\begin{equation}\label{alpha102000}
\alpha_{102}=\arccos\left(-1+\frac{(1-2B_{1})^{2}}{2B_{1}^{2}}\right)
\end{equation}
which yields $B_{1}>\frac{1}{4}.$

Taking into account (\ref{alpha203103}) and (\ref{alpha102000}),
we rotate the triangle $\triangle A_{1}A_{2}A_{3}$ about $A_{3}$
through $\pi-\alpha_{102}=2\alpha_{103}-\pi rad$ and we obtain the
triangle $\triangle A_{3}A_{1}^{\prime}A_{2}^{\prime}.$ Let
$A_{F}^{\prime}$ be the corresponding weighted Fermat-Torricelli
point of $\triangle A_{1}^{\prime}A_{2}^{\prime}A_{3},$ for
$B_{1}^{\prime}=B_{1}$ and $B_{2}=B_{2}^{\prime}$
(fig.~\ref{fig5}). Thus, the points $A_{2},$ $A_{F},$
$A_{F}^{\prime}$ and $A_{3}^{\prime}$ are collinear
(fig.~\ref{fig5}), because $\triangle A_{F}A_{F}^{\prime}A_{3}$ is
an isosceles triangle and \[\angle A_{F}A_{F}^{\prime}A_{3}=\angle
A_{F}^{\prime}A_{F}A_{3}=\pi-\alpha_{103}.\]

\end{proof}

%----------------------------------------------------------

\begin{corollary}{\cite{Hofmann:27}, \cite{Tikh:86},
\cite{BolMa/So:99}}\label{corhof} For $B_{1}=B_{2}=B_{3}$ the
solution of Problem~\ref{secgeomprob} is given by rotating the
$\triangle A_{1}A_{2}A_{3}$ about $A_{3}$ through $60^{\circ}.$
\end{corollary}

\begin{proof}
By replacing $B_{1}=B_{2}=B_{3}=1$ in the solution of
Problem~\ref{secgeomprob}, we deduce that the rotation about
$A_{3}$ need to be $\pi-120^{\circ}=60^{\circ}<\alpha_{132}$
(Hofmann's rotation).
\end{proof}

The author is sincerely grateful to Professor Dr. Vassilios G.
Papageorgiou for his very valuable comments, fruitful discussions
and for his permanent attention to this work.

\appendix

\section{}
We mention two methods of the length of a linear segment with
respect to (I) a variable length and (II) a variable angle, which
have been used, in order to find the weighted Fermat-Torricelli
point.

I. A method of differentiating the length of a linear segment with
respect to the length of a variable linear segment is given first
in \cite[Proposition~2.6, (b), Remark~2.4,
Corollary~3.3]{Zachos/Zou:08}, \cite[formula~(4),
p.~413]{Zachos/Zou:08b} and has been explained in detail in
\cite{Cots/Zach:11}, \cite[Corollary~2]{Zachos/Cots:10}.
Specifically, by differentiating (\ref{i23cos}) with respect to
$a_{1},$ and by replacing in the derived equation
$\cos(\alpha_{213}-\alpha_{013})$ taken from (\ref{i23cos}) , we
obtain:
\begin{equation}\label{dera2a1}
\frac{\partial a_2}{\partial a_{1}}=\cos(\alpha_{102}).
\end{equation}
Similarly, by differentiating (\ref{i23cos3}) with respect to
$a_{1},$ and by replacing in the derived equation
$\cos(\alpha_{013})$ taken from (\ref{i23cos3}) , we obtain:
\begin{equation}\label{dera3a1}
\frac{\partial a_3}{\partial a_{1}}=\cos(\alpha_{103}).
\end{equation}
We mention a method of differentiating the length of a linear
segment with respect to a variable angle, which have been used, in
order to find the weighted Fermat-Torricelli point
(\cite[Proposition~2.6 (b)]{Zachos/Zou:08}) in $\mathbb{R}^{2}.$
By mentioning this technique of differentiation, we correct some
typographical errors which appear in \cite{Zachos/Zou:08}.
Specifically, by differentiating (\ref{i23cos}) with respect to
$\alpha_{013},$ we get:

\begin{equation}\label{dera2013}
\frac{\partial a_{02}}{\partial
\alpha_{013}}=-a_{01}\frac{a_{12}}{a_{02}}\sin(\alpha_{213}-\alpha_{013})
\end{equation}
From the sine law in $\triangle A_{1}A_{0}A_{2},$ we get:

\begin{equation}\label{sin102}
\frac{a_{12}}{\sin(\alpha_{102})}=\frac{a_{02}}{\sin(\alpha_{213}-\alpha_{013})}
\end{equation}

By replacing (\ref{sin102}) in (\ref{dera2013}), we obtain:
\begin{equation}\label{dera2a013}
\frac{\partial a_{02}}{\partial
\alpha_{013}}=-a_{01}\sin(\alpha_{102}).
\end{equation}
Similarly, by differentiating (\ref{i23cos3}) with respect to
$\alpha_{013},$ we obtain:
\begin{equation}\label{dera3a013}
\frac{\partial a_{03}}{\partial
\alpha_{013}}=a_{01}\sin(\alpha_{103}).
\end{equation}

%\listoffigures
%\clearpage
\end{document}